\documentclass[11pt]{article}
\usepackage{amsmath,amsthm,amssymb}
\usepackage[dvips]{graphicx}

\newcommand{\F}{\mathbb{F}}
\newcommand{\N}{\mathbb{N}}

\renewcommand{\P}{\mathbb{P}}
\newcommand{\R}{\mathbb{R}}
\newcommand{\C}{\mathbb C}
\newcommand{\h}{{\cal H}}

\newcommand{\T}{{\cal T}}
\newcommand{\U}{{\cal U}}

\newcommand{\Tn}{\T_{n}^{-1}([-1,1])}
\newcommand{\tn}{\mathbb{T}_{n}}

\begin{document}

%----- Title -----%

\title{Inverse Polynomial Images which Consists of Two Jordan Arcs -- An Algebraic Solution\footnote{published in: Journal of Approximation Theory {\bf 148} (2007), 148--157.}}
\author{Klaus Schiefermayr\footnote{University of Applied Sciences Upper Austria, School of Engineering and Environmental Sciences, Stelzhamerstrasse\,23, 4600 Wels, Austria, \textsc{klaus.schiefermayr@fh-wels.at}}}
\date{}
\maketitle

\theoremstyle{plain}
\newtheorem{theorem}{Theorem}
\newtheorem{corollary}{Corollary}
\newtheorem{lemma}{Lemma}
\newtheorem{definition}{Definition}
\theoremstyle{definition}
\newtheorem{remark}{Remark}

\thispagestyle{empty}

\begin{abstract}
Inverse polynomial images of $[-1,1]$, which consists of two Jordan arcs, are characterised by an explicit polynomial equation for the four endpoints of the arcs.
\end{abstract}



%---------------------%
\section{Introduction}
%---------------------%

Let $\P_{n}$ be the set of polynomials with complex coefficients of degree $n$ and let $P_{n}\in\P_{n}$. Let $P_{n}^{-1}([-1,1])$ be the inverse image of $[-1,1]$ under the polynomial mapping $P_{n}$, i.e.,
\begin{equation}
P_{n}^{-1}([-1,1])=\bigl\{z\in\C:P_{n}(z)\in[-1,1]\bigr\}.
\end{equation}
In general, $P_{n}^{-1}([-1,1])$ consists of $n$ Jordan arcs, on which $P_{n}$ is strictly monotone increasing from $-1$ to $+1$, see \cite{PehInverseImage}. If there is a point $z_{1}\in\C$, for which $P_{n}(z_{1})\in\{-1,1\}$ and $P_{n}'(z_{1})=0$, then two Jordan arcs can be combined into one Jordan arc. This combination of arcs can be seen in a very good way from the inverse image of the classical Chebyshev polynomial $T_{n}(z)=\cos(n\arccos(z))$. The inverse image $T_{n}^{-1}([-1,1])$ is just $[-1,1]$, i.e.\ \emph{one} Jordan arc, since there are $n-1$ points $z_{j}$ with the property $T_{n}(z_{j})\in\{-1,1\}$ and $T_{n}'(z_{j})=0$. Note that $T_{n}$ is, up to a linear transformation, the only polynomial mapping with that property. In this paper, we are interested in polynomials for which the inverse image of $[-1,1]$ consists of two Jordan arcs. This property is equivalent to the existence of a certain quadratic equation for the corresponding polynomial.

%----- Definition: Tn-tuple -----%

\begin{definition}
The set $\{a,b,c,d\}$ of four complex numbers is called a $\tn$-tuple if there exist polynomials $\T_{n}\in\P_{n}$ and $\U_{n-2}\in\P_{n-2}$ such that a quadratic equation (sometimes called Pell-equation or Abel-equation) of the form
\begin{equation}\label{QuEqn}
\T_{n}^2(z)-\h(z)\,\U_{n-2}^2(z)=1
\end{equation}
holds, where
\begin{equation}\label{H}
\h(z):=(z-a)(z-b)(z-c)(z-d).
\end{equation}
Note that $\T_{n}$ and $\U_{n-2}$ are unique up to sign.
\end{definition}

%----- Text -----%

In \cite[Theorem\,3]{PehSch2004}, it was proved that a polynomial $\T_{n}\in\P_{n}$ satisfies equation \eqref{QuEqn} if and only if $\Tn$ consists of \emph{two} Jordan arcs with endpoints $a,b,c,d$.\\ The investigation of quadratic equations of the form \eqref{QuEqn} goes back to Abel\,\cite{Abel} and Chebyshev\,\cite{Chebyshev}. Their work was continued by Zolotarev\,\cite{Zolotarev1,Zolotarev2} and Achieser\,\cite{Achieser1929,Achieser1932}, and in recent times by Peherstorfer\,\cite{Peh_ActaMath,Peh1991,Peh1993,Peh1995}, Lebedev\,\cite{Lebedev} and Peherstorfer and the present author\,\cite{PehSch2004}. For a more detailed history on the subject and some pictures of $\Tn$ in several interesting cases, see \cite{PehSch2004}.

Most of the above cited papers make extensive use of elliptic functions and integrals. Nevertheless, our approach for characterising $\tn$-tuples is purely algebraic and is based on the following result of Peherstorfer and the author\,\cite[Lemma\,2.1]{PehSch1999-1}.

%----- Lemma: System of polynomial equations -----%

\begin{lemma}[Peherstorfer and Schiefermayr\,\cite{PehSch1999-1}]\hfill{}
\begin{enumerate}
\item Let $n=2m+1$ be an odd degree. The set $\{a,b,c,d\}$ is a $\tn$-tuple if and only if it satisfies the following system of equations:
\begin{equation}\label{odd}
\begin{aligned}
\bigl(x_{1}^{k}+\ldots+x_{m}^{k}\bigr)-\bigl(y_{1}^{k}+\ldots+y_{m-1}^{k}\bigr)
+\tfrac{1}{2}\bigl(-a^{k}-b^{k}-c^{k}+d^{k}\bigr)=0&,\\
k=1,2,\ldots,2m&.
\end{aligned}
\end{equation}
The corresponding polynomial $\T_{n}$ is given by
\begin{align*}
\T_{n}(z)&=1-\frac{2(z-d)\prod_{j=1}^{m}(z-x_{j})^{2}}{(a-d)\prod_{j=1}^{m}(a-x_{j})^{2}}\\
&=-1+\frac{2(z-a)(z-b)(z-c)\prod_{j=1}^{m-1}(z-y_{j})^{2}}{(d-a)(d-b)(d-c)\prod_{j=1}^{m-1}(d-y_{j})^{2}}.
\end{align*}
Note that $\T_{n}(x_{j})=\T_{n}(d)=1$ and $\T_{n}(y_{j})=\T_{n}(a)=\T_{n}(b)=\T_{n}(c)=-1$.
\item Let $n=2m+2$ be an even degree. The set $\{a,b,c,d\}$ is a $\tn$-tuple if and only if it satisfies one of the following two systems of equations:
\begin{equation}\label{even}
\begin{aligned}
\bigl(x_{1}^{k}+\ldots+x_{m}^{k}\bigr)-\bigl(y_{1}^{k}+\ldots+y_{m}^{k}\bigr)
+\tfrac{1}{2}\bigl(-a^{k}-b^{k}+c^{k}+d^{k}\bigr)=0&,\\
k=1,\ldots,2m+1&.
\end{aligned}
\end{equation}
or
\begin{equation}\label{even2}
\begin{aligned}
\bigl(x_{1}^{k}+\ldots+x_{m+1}^{k}\bigr)-\bigl(y_{1}^{k}+\ldots+y_{m-1}^{k}\bigr)+\tfrac{1}{2}\bigl(-a^{k}-b^{k}-c^{k}-d^{k}\bigr)=0&,\\
k=1,\ldots,2m+1&.
\end{aligned}
\end{equation}
The corresponding polynomial $\T_{n}$ for the solution of \eqref{even} is given by
\[
\T_{n}(z)=1-\frac{2(z-c)(z-d)\prod_{j=1}^{m}(z-x_{j})^{2}}{(a-c)(a-d)\prod_{j=1}^{m}(a-x_{j})^{2}}
=-1+\frac{2(z-a)(z-b)\prod_{j=1}^{m}(z-y_{j})^{2}}{(c-a)(c-b)\prod_{j=1}^{m}(c-y_{j})^{2}}.
\]
If $\{a,b,c,d\}$ satisfies system \eqref{even2}, i.e., is a $\tn$-tuple, then $\{a,b,c,d\}$ is also a $\mathbb{T}_{\frac{n}{2}}$-tuple and for the corresponding polynomial we have $\T_{n}(z)=2\T_{\frac{n}{2}}^{2}(z)-1$.
\end{enumerate}
\end{lemma}

%----- Text -----%

Note that the $x_{j}$ and $y_{j}$ of equation \eqref{odd}, \eqref{even}, and \eqref{even2}, are the zeros of the corresponding polynomial $\U_{n-2}$ and are therefore exactly the \emph{extremal points} of the corresponding polynomial $\T_{n}$ on $\Tn$ (which consists of two Jordan arcs). As usual, a point $z_{0}\in{C}$, $C\subseteq\C$ compact, is called an extremal point of $P_{n}\in\P_{n}$ on $C$ if $|P_{n}(z_{0})|=\max_{z\in{C}}|P_{n}(z)|$. Further, we want to remark that the above definition of a $\tn$-tuple also includes the case of one interval (if two of the four points $a,b,c,d$ are equal).

%----- Text -----%

The main purpose of the present paper is to modify the polynomial systems \eqref{odd} and \eqref{even} in the following way: With the help of the recent paper \cite{WuHadjicostis}, in Theorem\,\ref{Theorem-Odd} and Theorem\,\ref{Theorem-Even}, we give \emph{one} polynomial equation in terms of $a,b,c,d$, which is equivalent to \eqref{odd} and \eqref{even}, respectively. In other words, for every degree $n$, we can \emph{explicitly} give a polynomial in four variables $p(a,b,c,d)$, whose zeros $\{a,b,c,d\}$ are $\tn$-tuples. Moreover, a simple equation for computing the extremal points $x_{j}$ and $y_{j}$ is derived. Algebraic solutions of the quadratic equation \eqref{QuEqn} with the help of Jacobi's elliptic functions can be found in \cite{SoYu1993,SoYu1995} (in the real case), and \cite[Section\,4]{PehSch2004}. In \cite[Section\,5]{Peh_ActaMath}, an algebraic solution (in the real case) is given with the help of orthogonal polynomials. In \cite[Chapter\,7]{PehSch1999-2}, an algebraic solution was found by simplifying the above system of equations with the help of {\sc Gr\"obner}-Basis.

%----- Text -----%

The paper is organised as follows. In section\,2, the fundamental lemma based on \cite{WuHadjicostis} is given, from which the simplifications of \eqref{odd} and \eqref{even}, proved in section\,3, can be deduced. Moreover, the maximum number of $\tn$-tuples is given explicitly assuming that $3$ of the $4$ points $a,b,c,d$ are fixed. In section\,4.1, the polynomial equations for $a,b,c,d$ and for $x_{j}$ and $y_{j}$ are explicitly written down for the smallest degrees $n\in\{2,3,4\}$. Finally, a brief look at the special case of Zolotarev polynomials is taken in Section\,4.2.

%---------------------%
\section{Auxiliaries}
%---------------------%

Let us define $F_{0}:=1$ and, for $k=1,2,\ldots$,
\begin{equation}\label{Fk}
F_{k}\equiv{F}_{k}(s_{1},s_{2},\ldots,s_{k}):=\frac{(-1)^k}{k!}\det
\begin{pmatrix}
s_{1}&1&0&0&\hdots&0&0\\
s_{2}&s_{1}&2&0&\hdots&0&0\\
s_{3}&s_{2}&s_{1}&3&\hdots&0&0\\
\vdots&\vdots&\vdots&\vdots&\hdots&\vdots&\vdots\\
s_{k-1}&s_{k-2}&s_{k-3}&s_{k-4}&\hdots&s_{1}&k-1\\
s_{k}&s_{k-1}&s_{k-2}&s_{k-3}&\hdots&s_{2}&s_{1}
\end{pmatrix}.
\end{equation}
For negative indices $k=-1,-2,-3,\ldots$, we define $F_k:=0$.

The next lemma, which is fundamental for our considerations, can be extracted from \cite{WuHadjicostis}.

%----- Lemma: General solution -----%

\begin{lemma}\label{Lemma-WuHa}
For given $s_{1},s_{2},\ldots,s_{\nu+\mu}\in\C$, consider the system of equations
\begin{equation}\label{eqn-uv}
\bigl(u_{1}^{k}+\ldots+u_{\nu}^{k}\bigr)-\bigl(v_{1}^{k}+\ldots+v_{\mu}^{k}\bigr)=s_{k},\quad k=1,2,\ldots,\nu+\mu.
\end{equation}
If $\{u_{1},\ldots,u_{\nu},v_{1},\ldots,v_{\mu}\}$ is a nontrivial solution of \eqref{eqn-uv}, i.e., the sets $\{u_{1},\ldots,u_{\nu}\}$ and $\{v_{1},\ldots,v_{\mu}\}$ are disjoint, then this solution is unique (up to permutations of the $u_{j}$ and $v_{j}$) and the values $v_{1},v_{2},\ldots,v_{\mu}$ are exactly the solution of the equation
\begin{equation}\label{eqn-v}
v^{\mu}+\Lambda_{1}v^{\mu-1}+\Lambda_{2}v^{\mu-2}+\ldots+\Lambda_{\mu-1}v+\Lambda_{\mu}=0,
\end{equation}
where $\Lambda_{1},\Lambda_{2},\ldots,\Lambda_{\mu}$ are the solution of the following regular linear system of equations:
\begin{equation}\label{eqn-lambda}
\begin{aligned}
\sum_{i=1}^{\mu}F_{\nu+1-i}\Lambda_{i}&=-F_{\nu+1}\\
\sum_{i=1}^{\mu}F_{\nu+2-i}\Lambda_{i}&=-F_{\nu+2}\\
&\vdots\\
\sum_{i=1}^{\mu}F_{\nu+\mu-i}\Lambda_{i}&=-F_{\nu+\mu}
\end{aligned}
\end{equation}
By Cramer's rule, the solution $\Lambda_{1},\Lambda_{2},\ldots,\Lambda_{\mu}$ of system \eqref{eqn-lambda} may be written in the form (note that system \eqref{eqn-lambda} is regular and therefore $\det\F\neq0$)
\begin{equation}\label{Lambda}
\Lambda_{i}=\frac{\det\F_{i}}{\det\F}, \qquad i=1,2,\ldots,\mu,
\end{equation}
where
\begin{equation}\label{FF}
\F:=\begin{pmatrix}
F_{\nu}&F_{\nu-1}&F_{\nu-2}&\hdots&F_{\nu+1-\mu}\\
F_{\nu+1}&F_{\nu}&F_{\nu-1}&\hdots&F_{\nu+2-\mu}\\
F_{\nu+2}&F_{\nu+1}&F_{\nu}&\hdots&F_{\nu+3-\mu}\\
\vdots&\vdots&\vdots&\hdots&\vdots\\
F_{\nu+\mu-1}&F_{\nu+\mu-2}&F_{\nu+\mu-3}&\hdots&F_{\nu}
\end{pmatrix}\in\R_{\mu}^{\mu}
\end{equation}
and
\begin{equation}\label{FFi}
\F_{i}:=\F \text{ and the $i$-th column of $\F$ is replaced by }
\begin{pmatrix}-F_{\nu+1}\\-F_{\nu+2}\\-F_{\nu+3}\\\vdots\\-F_{\nu+\mu}\end{pmatrix}.
\end{equation}
By \eqref{Lambda}, equation \eqref{eqn-v} may be written in the form
\begin{equation}
v^{\mu}\det\F+v^{\mu-1}\det\F_{1}+v^{\mu-2}\det\F_{2}
+\ldots+v\det\F_{\mu-1}+\det\F_{\mu}=0.
\end{equation}
\end{lemma}

%---------------------%
\section{Main Results}
%---------------------%

%----- odd degree: part 1 -----%

Let $n=2m+1$ be an odd degree. Starting point is system of equations \eqref{odd}, which may be written in the form
\begin{equation}\label{odd2}
\bigl(y_{1}^{k}+\ldots+y_{m-1}^{k}\bigr)
-\bigl(x_{1}^{k}+\ldots+x_{m}^{k}+d^{k}\bigr)=s_{k}, \quad
k=1,2,\ldots,2m,
\end{equation}
where
\begin{equation}\label{odd-sk1}
s_{k}:=\tfrac{1}{2}\bigl(-a^{k}-b^{k}-c^{k}-d^{k}\bigr), \qquad k=1,2,\ldots,2m.
\end{equation}
Then, by Lemma\,\ref{Lemma-WuHa}, the above polynomial system reduces to
\[
d^{m+1}\det\F+d^{m}\det\F_{1}+d^{m-1}\det\F_{2}+\ldots+d\det\F_{m}+\det\F_{m+1}=0,
\]
where $s_{k}$, $F_{k}$, $\F$, and $\F_{i}$ is defined in \eqref{odd-sk1}, \eqref{Fk}, \eqref{FF}, and \eqref{FFi}, respectively, and $\nu=m-1$ and $\mu=m+1$. Note that if $\{y_{1},\ldots,y_{m-1},x_{1},\ldots,x_{m},d\}$ is a solution of \eqref{odd2} then the sets $\{y_{1},\ldots,y_{m-1}\}$ and $\{x_{1},\ldots,x_{m},d\}$ are disjoint, thus, by Lemma\,\ref{Lemma-WuHa}, the corresponding linear system \eqref{eqn-lambda} is regular and $\det\F\neq0$.

%----- odd degree: part 2 -----%

In order to get a polynomial equation for the $y_{i}$, we write system \eqref{odd} in the form
\begin{equation}
\bigl(x_{1}^{k}+\ldots+x_{m}^{k}+d^{k}\bigr)
-\bigl(y_{1}^{k}+\ldots+y_{m-1}^{k}\bigr)=s_{k}, \quad
k=1,2,\ldots,2m,
\end{equation}
where
\begin{equation}\label{odd-sk2}
s_{k}:=\tfrac{1}{2}\bigl(a^{k}+b^{k}+c^{k}+d^{k}\bigr), \qquad k=1,2,\ldots,2m.
\end{equation}
Then, again by Lemma\,\ref{Lemma-WuHa}, the above polynomial system reduces to
\[
y^{m-1}\det\F+y^{m-2}\det\F_{1}+y^{m-3}\det\F_{2}+\ldots+y\det\F_{m-2}+\det\F_{m-1}=0,
\]
where $s_{k}$, $F_{k}$, $\F$, and $\F_{i}$ is defined in \eqref{odd-sk2}, \eqref{Fk}, \eqref{FF}, and \eqref{FFi}, respectively, and $\nu=m+1$, $\mu=m-1$. By the same argument as above, $\det\F\neq0$.

%----- odd degree: part 3 -----%

For a polynomial equation for the $x_{i}$, we write system \eqref{odd} in the form
\begin{equation}
\bigl(y_{1}^{k}+\ldots+y_{m-1}^{k}+a^{k}\bigr)
-\bigl(x_{1}^{k}+\ldots+x_{m}^{k}\bigr)=s_{k}, \quad
k=1,2,\ldots,2m,
\end{equation}
where
\begin{equation}\label{odd-sk3}
s_{k}:=\tfrac{1}{2}\bigl(a^{k}-b^{k}-c^{k}+d^{k}\bigr), \qquad k=1,2,\ldots,2m.
\end{equation}
Then, again by Lemma\,\ref{Lemma-WuHa}, the above polynomial system reduces to
\[
x^{m}\det\F+x^{m-1}\det\F_{1}+x^{m-2}\det\F_{2}+\ldots+x\det\F_{m-1}+\det\F_{m}=0,
\]
where $s_{k}$, $F_{k}$, $\F$, and $\F_{i}$ is defined in \eqref{odd-sk3}, \eqref{Fk}, \eqref{FF}, and \eqref{FFi}, respectively, and $\nu=m$, $\mu=m$. By the same argument as above, $\det\F\neq0$.

We collect the above results in the following theorem.

%----- Theorem: odd case -----%

\begin{theorem}\label{Theorem-Odd}
Let $n=2m+1$.
\begin{enumerate}
\item The set $\{a,b,c,d\}$ is a $\tn$-tuple if and only if $a,b,c,d$ satisfies the polynomial equation
\begin{equation}\label{odd-abcd}
p:=d^{m+1}\det\F+d^{m}\det\F_{1}+d^{m-1}\det\F_{2}+\ldots+d\det\F_{m}+\det\F_{m+1}=0,
\end{equation}
where $s_{k}$, $F_{k}$, $\F$, and $\F_{i}$ is defined in \eqref{odd-sk1}, \eqref{Fk}, \eqref{FF}, and \eqref{FFi}, respectively, and $\nu=m-1$, $\mu=m+1$.
\item The values $y_{1},y_{2},\ldots,y_{m-1}$ are exactly the zeros of the polynomial
\begin{equation}\label{odd-y}
y^{m-1}\det\F+y^{m-2}\det\F_{1}+y^{m-3}\det\F_{2}+\ldots+y\det\F_{m-2}+\det\F_{m-1},
\end{equation}
where $s_{k}$, $F_{k}$, $\F$, and $\F_{i}$ is defined in \eqref{odd-sk2}, \eqref{Fk}, \eqref{FF}, and \eqref{FFi}, respectively, and $\nu=m+1$, $\mu=m-1$.
\item The values $x_{1},x_{2},\ldots,x_{m}$ are exactly the zeros of the polynomial
\begin{equation}\label{odd-x}
x^{m}\det\F+x^{m-1}\det\F_{1}+x^{m-2}\det\F_{2}+\ldots+x\det\F_{m-1}+\det\F_{m},
\end{equation}
where $s_{k}$, $F_{k}$, $\F$, and $\F_{i}$ is defined in \eqref{odd-sk3}, \eqref{Fk}, \eqref{FF}, and \eqref{FFi}, respectively, and $\nu=m$, $\mu=m$.
\end{enumerate}
\end{theorem}

%----- Corollary: odd case -----%

\begin{corollary}\label{Corollary-Odd}
Let $n=2m+1$.
\begin{enumerate}
\item The polynomial $p\equiv{p}(a,b,c,d)$ in Theorem\,\ref{Theorem-Odd}(i) is a homogeneous polynomial of $a,b,c,d$ with rational coefficients and degree $m^{2}+m=(n^{2}-1)/4$.
\item Let $3$ of the $4$ points $a,b,c,d\in\C$ be fixed, then there exist at most $(n^{2}-1)$ $\tn$-tuples containing these $3$ points.
\end{enumerate}
\end{corollary}
\begin{proof}
\begin{enumerate}
\item By the definitions \eqref{odd-sk1}, \eqref{Fk}, \eqref{FF}, and \eqref{FFi}, the following statements concerning the degree of $p$ hold (note that $\nu=m-1$, $\mu=m+1$):
\begin{itemize}
\item $F_k$ is a homogeneous polynomial of $a,b,c,d$ with degree $k$.
\item $\det\F$ is a homogeneous polynomial of $a,b,c,d$ with degree $m^{2}-1$.
\item $\det\F_{i}$ is a homogeneous polynomial of $a,b,c,d$ with degree $m^{2}-1+i$.
\end{itemize}
From these statements, the assertion follows.
\item By the special form of equation \eqref{odd}, there are $4$ different possibilities to fix $3$ of the $4$ points $a,b,c,d$ in \eqref{odd}. These $4$ possibilities multiplied with the degree $(n^{2}-1)/4$ of the homogeneous polynomial $p(a,b,c,d)$ gives the maximum number of different solutions.
\end{enumerate}
\end{proof}

Finally, we give the analogous results for even degree. Since the proofs run along the same lines as those for odd degree, we omit them.

%----- Theorem: even case -----%

\begin{theorem}\label{Theorem-Even}
Let $n=2m+2$.
\begin{enumerate}
\item The set $\{a,b,c,d\}$ is a $\tn$-tuple but not a $\mathbb{T}_{\frac{n}{2}}$-tuple if and only if $a,b,c,d$ satisfies the polynomial equation
\begin{equation}\label{even-abcd}
p:=a^{m+1}\det\F+a^{m}\det\F_{1}+a^{m-1}\det\F_{2}+\ldots+a\det\F_{m}+\det\F_{m+1}=0,
\end{equation}
where
\begin{equation}\label{even-sk1}
s_{k}:=\tfrac{1}{2}\bigl(-a^{k}+b^{k}-c^{k}-d^{k}\bigr), \qquad k=1,2,\ldots,2m+1,
\end{equation}
and $F_{k}$, $\F$, and $\F_{i}$ is defined in \eqref{Fk}, \eqref{FF}, and \eqref{FFi}, respectively, and $\nu=m$, $\mu=m+1$.
\item The values $y_{1},y_{2},\ldots,y_{m}$ are exactly the zeros of the polynomial
\begin{equation}\label{even-y}
y^{m}\det\F+y^{m-1}\det\F_{1}+y^{m-2}\det\F_{2}+\ldots+y\det\F_{m-1}+\det\F_{m},
\end{equation}
where
\begin{equation}\label{even-sk2}
s_{k}:=\tfrac{1}{2}\bigl(a^{k}+b^{k}+c^{k}-d^{k}\bigr), \qquad k=1,2,\ldots,2m+1,
\end{equation}
and $F_{k}$, $\F$, and $\F_{i}$ is defined in \eqref{Fk}, \eqref{FF}, and \eqref{FFi}, respectively, and $\nu=m+1$, $\mu=m$.
\item The values $x_{1},x_{2},\ldots,x_{m}$ are exactly the zeros of the polynomial
\begin{equation}\label{even-x}
x^{m}\det\F+x^{m-1}\det\F_{1}+x^{m-2}\det\F_{2}+\ldots+x\det\F_{m-1}+\det\F_{m},
\end{equation}
where
\begin{equation}\label{even-sk3}
s_{k}:=\tfrac{1}{2}\bigl(a^{k}-b^{k}+c^{k}+d^{k}\bigr), \qquad k=1,2,\ldots,2m+1,
\end{equation}
and $F_{k}$, $\F$, and $\F_{i}$ is defined in \eqref{Fk}, \eqref{FF}, and \eqref{FFi}, respectively, and $\nu=m+1$, $\mu=m$.
\end{enumerate}
\end{theorem}

%----- Corollary: even case -----%

\begin{corollary}
Let $n=2m+2$.
\begin{enumerate}
\item The polynomial $p\equiv{p}(a,b,c,d)$ in Theorem\,\ref{Theorem-Even}(i) is a homogeneous polynomial of $a,b,c,d$ with rational coefficients and degree $(m+1)^{2}=n^{2}/4$.
\item Let $3$ of the $4$ points $a,b,c,d\in\C$ be fixed, then there exist at most $3n^{2}/4$ $\tn$-tuples containing these $3$ points.
\end{enumerate}
\end{corollary}

%----------------%
\section{Addenda}
%----------------%

%---------------------------------------%
\subsection{Equations for Small Degrees}
%---------------------------------------%

%----- Text -----%

By Theorems\,\ref{Theorem-Odd} and \ref{Theorem-Even}, the computation of a $\tn$-tuple $\{a,b,c,d\}$ and the extremal points $x_{i}$ and $y_{i}$ of the corresponding polynomial $\T_{n}$ can be managed in $4$ steps:
\begin{enumerate}
\item[1.] Given $3$ of the $4$ points $a,b,c,d$, compute the $4^{\text{th}}$ point by equation \eqref{odd-abcd} and \eqref{even-abcd}, respectively.
\item[2.] Compute the $y_{j}$ by equation \eqref{odd-y} and \eqref{even-y}, respectively.
\item[3.] Compute the $x_{j}$ by equation \eqref{odd-x} and \eqref{even-x}, respectively.
\item[4.] Compute the corresponding polynomial $\T_{n}$ by Lemma\,1.
\end{enumerate}
In the following, we give the equations for $a,b,c,d$, for the $y_{j}$, and for the $x_{j}$, in case of the simplest degrees $n=2,3,4$. For greater degrees, the equations get very bulky.
\begin{itemize}
%----- n=2 -----%
\item $n=2$:
\[
a+b-c-d=0
\]
%----- n=3 -----%
\item $n=3$:
\[
a^2-2ab+b^2-2ac-2bc+c^2+2ad+2bd+2cd-3d^2=0
\]
\[
x(-4a+4b+4c-4d)+(a^2+2ab-3b^2+2ac-2bc-3c^2-2ad+2bd+2cd+d^2)=0
\]
%----- n=4 -----%
\item $n=4$:
\begin{align*}
&a^4+4a^3b-10a^2b^2+4ab^3+b^4-4a^3c+4a^2bc+4ab^2c-4b^3c+6a^2c^2\\
&-4abc^2+6b^2c^2-4ac^3-4bc^3+c^4-4a^3d+4a^2bd+4ab^2d-4b^3d-4a^2cd\\
&-8abcd-4b^2cd+4ac^2d+4bc^2d+4c^3d+6a^2d^2-4abd^2+6b^2d^2\\
&+4acd^2+4bcd^2-10c^2d^2-4ad^3-4bd^3+4cd^3+d^4=0
\end{align*}
\begin{align*}
&y(-2a^2+4ab-2b^2+4ac+4bc-2c^2-4ad-4bd-4cd+6d^2)\\
&+(a^3-a^2b-ab^2+b^3-a^2c+2abc-b^2c-ac^2-bc^2+c^3+a^2d\\
&-2abd+b^2d- 2acd- 2bcd+c^2d+3ad^2+3bd^2+3cd^2-5d^3)=0
\end{align*}
\begin{align*}
&x(-2a^2-4ab+6b^2+4ac-4bc-2c^2+4ad-4bd+4cd-2d^2)\\
&+(a^3+a^2b+3ab^2-5b^3-a^2c-2abc+3b^2c-ac^2+bc^2+c^3\\
&-a^2d-2abd+ 3 b^2d+2acd-2bcd-c^2d-ad^2+bd^2-cd^2+d^3)=0
\end{align*}
\end{itemize}

%----------------------------------------------%
\subsection{Special Case: Zolotarev Polynomial}
%----------------------------------------------%

A special case for which the inverse polynomial image consists of two arcs is the so-called Zolotarev polynomial, which has also applications in signal processing \cite{Unbehauen}. Given $\sigma>0$, the Zolotarev polynomial $Z_{n}(x)$ solves the following approximation problem \cite[Addendum\,E]{AchieserBook}:
\[
\min_{a_{j}\in\C}\max_{x\in[-1,1]}\bigl|x^{n}-n\sigma{x}^{n-1}+a_{n-2}x^{n-2}+\ldots+a_{1}x+a_{0}\bigr|
=\max_{x\in[-1,1]}\bigl|Z_{n}(x)\bigr|=:L_{n}
\]
It is well known that, for $\sigma\leq\tan^{2}\frac{\pi}{2n}$, the Zolotarev polynomial $Z_{n}(x)$ is simply a suitable linear transformed classical Chebyshev polynomial $T_{n}(x)$ and, for $\sigma>\tan^{2}(\frac{\pi}{2n})$, the inverse image of $Z_{n}(x)$ consists of the two arcs $[-1,1]\cup[\alpha,\beta]$, where $1<\alpha<\beta$ and $Z_{n}$ has $n$ and $2$ extremal points in $[-1,1]$ and $[\alpha,\beta]$, respectively. The connection of the parameter $\sigma$ with the extremal points of $Z_{n}$ follows from the well known theorem of Vieta applied to the polynomial $Z_{n}(x)+L_{n}$, which in our notation (put $a=\alpha$, $b=1$, $c=-1$, $d=\beta$) reads as follows:
\begin{align}
2\sum_{j=1}^{m-1}y_{j}+\alpha&=n\sigma \qquad (n=2m+1) \label{Vieta_odd}\\
2\sum_{j=1}^{m}y_{j}+\alpha+1&=n\sigma \qquad (n=2m+2)
\label{Vieta_even}
\end{align}
We summarize the results in the following corollary.

%----- Corollary: Zolotarev case -----%

\begin{corollary}
Given $n\in\N$ and $\sigma>\tan^{2}(\frac{\pi}{2n})$, the two endpoints $\alpha$ and $\beta$ of the inverse polynomial image of the corresponding Zolotarev polynomial can be computed by the following two polynomial equations:
\begin{enumerate}
\item $n=2m+1$:
\begin{align}
\beta^{m+1}\det\F+\beta^{m}\det\F_{1}+\beta^{m-1}\det\F_{2}+\ldots+\beta\det\F_{m}+\det\F_{m+1}&=0 \label{Zol_odd1}\\
-2\det\F_{1}+(\alpha-n\sigma)\det\F&=0 \label{Zol_odd2}
\end{align}
where  $F_{k}$, $\F$, $\F_{i}$ is defined in \eqref{Fk}, \eqref{FF}, \eqref{FFi}, respectively, and
\[
\begin{array}{lll}
\nu=m-1,\mu=m+1,\quad&s_{k}:=\tfrac{1}{2}(-\alpha^{k}-1-(-1)^{k}-\beta^{k})\quad&\text{in}\quad\eqref{Zol_odd1}\\
\nu=m+1,\mu=m-1,&s_{k}:=\tfrac{1}{2}(\alpha^{k}+1+(-1)^{k}+\beta^{k})&\text{in}\quad\eqref{Zol_odd2}
\end{array}
\]
\item $n=2m+2$:
\begin{align}
\alpha^{m+1}\det\F+\alpha^{m}\det\F_{1}+\alpha^{m-1}\det\F_{2}+\ldots+\alpha\det\F_{m}+\det\F_{m+1}&=0 \label{Zol_even1}\\
-2\det\F_{1}+(\alpha+1-n\sigma)\det\F&=0 \label{Zol_even2}
\end{align}
where  $F_{k}$, $\F$, $\F_{i}$ is defined in \eqref{Fk}, \eqref{FF}, \eqref{FFi}, respectively, and
\[
\begin{array}{lll}
\nu=m,\mu=m+1,\quad&s_{k}:=\tfrac{1}{2}(-\alpha^{k}+1-(-1)^{k}-\beta^{k})\quad&\text{in}\quad\eqref{Zol_even1}\\
\nu=m+1,\mu=m,&s_{k}:=\tfrac{1}{2}(\alpha^{k}+1+(-1)^{k}-\beta^{k})&\text{in}\quad\eqref{Zol_even2}
\end{array}
\]
\end{enumerate}
\end{corollary}
\begin{proof}
Equation \eqref{Zol_odd1} and \eqref{Zol_even1} follow immediately by Theorem\,\ref{Theorem-Odd}(i) and Theorem\,\ref{Theorem-Even}(i), respectively. Equation \eqref{Zol_odd2} and \eqref{Zol_even2} follow from \eqref{Vieta_odd} and \eqref{Vieta_even} together with Theorem\,\ref{Theorem-Odd}(ii) and Theorem\,\ref{Theorem-Even}(ii), respectively, using Vieta's root theorem.
\end{proof}

%----- Remark -----%

\begin{remark}\hfill{}
\begin{enumerate}
\item Let $n\in\N$ and $\sigma>\tan^{2}(\frac{\pi}{2n})$. With the equations \eqref{Zol_odd1},\eqref{Zol_odd2}, and \eqref{Zol_even1},\eqref{Zol_even2}, respectively, one can determine $\alpha,\beta$ uniquely such that $1<\alpha<\beta$.
\item Once $\alpha,\beta$ are determined, by Theorem\,\ref{Theorem-Odd}(ii),(iii) and Theorem\,\ref{Theorem-Even}(ii),(iii), respectively, one can easily compute the extremal points $x_{j}$ and $y_{j}$ (which are all in $[-1,1]$).
\item With $\alpha$, $\beta$, the $x_{j}$, and the $y_{j}$, the corresponding Zolotarev polynomial ist given, for $n=2m+1$ odd,
\begin{equation}
\begin{aligned}
Z_{n}(x)&=(x-\beta)\prod_{j=1}^{m}(x-x_{j})^{2}
+\tfrac{1}{2}(\beta-\alpha)\prod_{j=1}^{m}(\alpha-x_{j})^{2}\\
&=(x-\alpha)(x^{2}-1)\prod_{j=1}^{m-1}(x-y_{j})^{2}
-\tfrac{1}{2}(\beta-\alpha)(\beta^{2}-1)\prod_{j=1}^{m-1}(\beta-y_{j})^{2},
\end{aligned}
\end{equation}
and, for $n=2m+2$ even,
\begin{equation}
\begin{aligned}
Z_{n}(x)&=(x+1)(x-\beta)\prod_{j=1}^{m}(x-x_{j})^{2}+\tfrac{1}{2}(\alpha+1)(\beta-\alpha)\prod_{j=1}^{m}(\alpha-x_{j})^{2}\\
&=(x-1)(x-\alpha)(x^{2}-1)\prod_{j=1}^{m-1}(x-y_{j})^{2}\\
&\qquad-\tfrac{1}{2}(\beta-1)(\beta-\alpha)(\beta^{2}-1)\prod_{j=1}^{m-1}(\beta-y_{j})^{2}.
\end{aligned}
\end{equation}
\item For a different approach to an algebraic solution of the Zolotarev problem, see \cite{Malyshev2002,Malyshev2003}.
\end{enumerate}
\end{remark}

%----- Bibliography -----%

\bibliographystyle{amsplain}

\bibliography{AlgebraicSolution}

\end{document}